\documentclass[11pt,a4paper, leqno]{article} 
\usepackage[centertags]{amsmath}
\usepackage{amsfonts}
\usepackage{amssymb}
\usepackage{amsthm}
\usepackage{epsfig}
\usepackage{setspace}
\usepackage{comment}
\usepackage{ae} 
\usepackage[authoryear]{natbib}
\usepackage{graphicx}

\theoremstyle{plain}
\newtheorem{thm}{Theorem}[section]
\newtheorem{lem}[thm]{Lemma}
\newtheorem{cor}[thm]{Corollary}
\newtheorem{prop}[thm]{Proposition}
\theoremstyle{definition}

\theoremstyle{remark}

\newtheorem{rem}[thm]{Remark}

\newcommand{\E}{\operatorname{E}}
\newcommand{\Var}{\operatorname{Var}}
\renewcommand{\P}{\operatorname{P}}

\newcommand{\BCE}{\boldsymbol{\mathcal{S}}}
\newcommand{\CE}{\mathcal{S}}

\newcommand{\keywords}[1]{\par\noindent\emph{Keywords:} #1 \\}

\begin{document}
\title{On the rates of convergence of simulation based optimization algorithms  for optimal stopping problems}
\author{Denis Belomestny$^{1,\,}$\thanks{supported in part by the SFB 649 `Economic Risk'.
}}
\footnotetext[1]{Weierstrass Institute for Applied Analysis and
Stochastics, Mohrenstr. 39, 10117 Berlin, Germany.
{\tt{belomest@wias-berlin.de}}. }
\maketitle
\begin{abstract}
In this paper we  study simulation based optimization algorithms for solving discrete time optimal stopping problems. This type of algorithms became  popular among practioneers working in the area of quantitative finance. Using large deviation theory for the increments of empirical processes, we derive optimal convergence rates  and show that they can not be improved in general.  The rates derived  provide a guide to the choice of the number of  simulated paths needed in optimization step, which is crucial for the good performance of  any simulation based optimization algorithm. Finally, we present a numerical example of solving optimal stopping problem arising in option pricing that illustrates our theoretical findings.
\end{abstract}
\section{Introduction}
The theory of optimal stopping is concerned with the problem of choosing a time to take a particular action, in order to maximise an expected reward or minimise an expected cost. Optimal stopping problems can be found in many areas of statistics, economics, and mathematical finance. They can often be written in the form of a Bellman equation, and are therefore often solved using dynamic programming. Results on optimal stopping were first developed in the discrete case. The formulation of optimal
stopping problems for discrete stochastic processes was in sequential analysis, an area of
mathematical statistics where the number of observations is not fixed in advance but is a random
number determined by the behavior of the data being observed.
\citet{SN} was the first person to come up with results on optimal stopping theory for stochastic
processes in discrete time.
We refer to the book of \citet{PS} for a comprehensive review on different  aspects of optimal stopping problems.
\par
A huge impetus to the development of optimal stopping theory was provided by option pricing theory, developed in the late 1960s and the 1970s. According to the modern financial theory, pricing an American option in a complete market is equivalent to solving an optimal stopping problem (with a corresponding generalization in incomplete markets), the optimal stopping time being the rational time for the option to be exercised. Due to the enormous importance of the early exercise feature in finance, this line of research has been intensively pursued in recent times.
Solving the optimal stopping problem and hence pricing an American option is
straightforward in low dimensions. However, many problems arising
in practice  have high dimensions, and these applications
have motivated the development of Monte Carlo methods for pricing American option.
Solving  a high-dimensional optimal stopping problems or pricing American style derivatives with Monte Carlo is a challenging task because
the determination of the optimal value function  requires a backwards dynamic
programming algorithm that appears to be incompatible with the forward nature
of Monte Carlo simulation. Much research was focused on the development of
fast methods to compute approximations to the optimal value function.
Notable examples include
mesh method of \citet{BG}, the regression-based approaches of \citet{Car}, \citet{LS}, \citet{TV} and \citet{E}.
All these methods aim at approximating the so called continuation values that can be used later
to construct suboptimal strategies and to produce lower bounds for the optimal value function.
The convergence analysis for this type of methods was performed in several papers including
\citet{E}, \cite{EKT} and \citet{B}.
An alternative to trying to approximate the continuation values is to find the best
value function within a class of stopping rules. This reduces
the optimal stopping problem to a much more tractable finite dimensional
optimization problem. Such optimization
problems appear naturally if one considers finite dimensional or parametric
approximations for the corresponding stopping regions. The latter type of algorithms became
particularly popular among practioneers (see e.g. \citet{A} or \citet{G}).
However, the practical success of simulation-based optimization algorithms
has not been yet fully explained by existing theory, and our analysis here
represents a further step toward an improved understanding. The main goal of this work is to provide rigorous convergence analysis of simulation based optimization algorithms
for discrete time optimal stopping problems.
\par
Let us start with a general stochastic programming problem
\begin{eqnarray}
\label{SPP}
    h^{*}:=\min_{\theta\in \Theta} \E_{\P}[h(\theta,\xi)],
\end{eqnarray}
where \( \Theta \) is a subset of \( \mathbb{R}^{m} \), \( \xi  \) is a \( \mathbb{R}^{d} \) valued  random variable on the probability space \( (\Omega,\mathcal{F}, \P) \) and \( h: \mathbb{R}^{m}\times \mathbb{R}^{d}\to \mathbb{R}.\)
Draw an i.i.d. sample \( \xi^{(1)},\ldots,\xi^{(M)} \) from the distribution of \( \xi  \) and define
\begin{eqnarray*}
    h_{M}:=\min_{\theta\in \Theta}\left[ \frac{1}{M}\sum_{m=1}^{M} h(\theta,\xi^{(m)})\right].
\end{eqnarray*}
It is well known (see e.g. \citet{S0})  that under very mild conditions it holds \( h_{M}-h^{*}=O_{\P}(M^{-1/2}). \)
In their pioneering work \citet{SH} (see also \citet{KSH}) showed that in the case of  discrete random variable \( \xi,  \) the convergence of \( h_{N} \) to \( h^{*} \) can be much faster than \( M^{-1/2}, \) making Monte Carlo method particularly efficient in this situation.   Turn now to  the discrete time optimal stopping problem:
\begin{eqnarray}
\label{OSPI}
    V=\sup_{1\leq \tau\leq K} \E [Z_{\tau }],
\end{eqnarray}
where \( \tau  \) is a stopping time taking values in the set \( \{ 1,\ldots, K \} \) and \( (Z_{k})_{k\geq 0} \) is a Markov chain.
Since the random variable  \( \tau  \) takes only discrete values, one can ask whether the  simulation based methods in the case of discrete time optimal stopping problem \eqref{OSPI}  can be
as efficient as in the case of \eqref{SPP} with discrete r.v. \( \xi  \). In this work we give an affirmative answer to this question
by deriving the optimal rates of convergence for the corresponding Monte Carlo estimate of \( V \) based on \( M \) paths and  showing that these rates are usually faster than  \( M^{-1/2} \).

\section{Main setup}
\label{MS}
Let us consider a Markov chain \( X = (X_k)_{k\geq 0} \)
defined on a filtered probability space \( (\Omega,\mathcal{F},(\mathcal{F}_{k})_{k\geq 0},\P_{x})\)
and taking values in a measurable space \( (E,\mathcal{B}), \) where for simplicity we assume that
\( E=\mathbb{R}^{d} \) for some \( d\geq 1 \) and \( \mathcal{B}=\mathcal{B}(\mathbb{R}^{d}) \)
is the Borel \( \sigma  \)-algebra on \( \mathbb{R}^{d}. \) It is assumed
that the chain \( X \) starts at \( x \) under \( \P_x \) for some \( x\in E \) . We also assume that
the mapping \(  x\mapsto P_x(A) \) is measurable for each \( A\in \mathcal{F} \) .
Fix some natural number \( K>0. \) Given a set of measurable functions \( G_{k}: E\mapsto \mathbb{R} \), \( k=1,\ldots, K, \) satisfying
\begin{eqnarray*}
    \E_{x}\left[ \sup_{1\leq k\leq K}|G_{k}(X_{k})| \right]<\infty
\end{eqnarray*}
for all \( x\in E \) , we consider the optimal stopping problems
\begin{eqnarray}
\label{OSP}
    V^{*}_{k}(x):=\sup_{k\leq \tau\leq K }\E_{k,\, x}\left[ G_{\tau}(X_{\tau}) \right], \quad k=1,\ldots,K,
\end{eqnarray}
where for any \( x\in E \) the expectation in \eqref{OSP} is taken w.r.t.  the measure \( \P_{k,\, x} \) such that \( X_k = x \) under
 \( \P_{k,\, x} \) and the supremum is taken over all stopping times \( \tau  \) with respect to \( (\mathcal{F}_{n})_{n\geq 0}. \)
Introduce the stopping region \( \BCE^{*}=\CE^{*}_{1}\times\ldots\times \CE^{*}_{K} \)  with
\( \mathcal{S}^{*}_{K}=E \) and
\begin{multline*}
    \mathcal{S}^{*}_{k}:=\{ x\in E: V^{*}_{k}(x)=G_{k}(x) \}
    \\
    =\left\{ x\in E: \E\left[  \left.V^{*}_{k+1}(X_{k+1})\right|\mathcal{F}_{k} \right]\leq G_{k}(x) \right\}, \quad k=1,\ldots, K-1.
\end{multline*}
Introduce also the first entry times \( \tau^{*}_{k}  \) into \( \BCE^{*} \)
by setting
\begin{eqnarray*}
    \tau^{*}_{k}:=\tau_{k}(\BCE^{*}):=\min \{ k\leq l \leq K: X_{l}\in \mathcal{S}_{l} \}.
\end{eqnarray*}
It is well known that the value functions \( V^{*}_{k}(x) \) satisfy the so called Wald-Bellman equations
\begin{eqnarray*}
    V^{*}_{k}(x)=\max\{ G_{k}(x), \E_{n,x}[V^{*}_{k+1}(X_{k+1})] \}, \quad k=1,\ldots, K-1, \quad x\in E
\end{eqnarray*}
with \( V^{*}_{K}(x)\equiv G_{K}(x) \)  by definition. Moreover, the stopping times \( \tau^{*}_{k}  \) are optimal
in \eqref{OSP}, i.e.
\begin{eqnarray*}
    V^{*}_{k}(x)=\E_{k,\, x}\left[ G_{\tau^{*}_{k}}(X_{\tau^{*}_{k}}) \right], \quad k=1,\ldots,K.
\end{eqnarray*}
\par
Let \( (X^{(m)}_{k})_{k=0,\ldots,K}, \, m=1,\ldots,M \) be \( M \) independent processes with the same distribution
as \(X \) all starting from the point \( x\in E. \)
We can think of \( (X^{(1)}_{k},\ldots, X^{(M)}_{k}),\) \(k=0,\ldots,K, \) as a new process defined on the product
probability space equipped with the product measure \( \P_{x}^{\otimes M}. \)
 Let
\( \mathfrak{B} \) be a collection of sets
from the product \( \sigma  \)-algebra
\[
 \mathcal{B}^{K}:=\underbrace{\mathcal{B}\otimes\ldots \otimes \mathcal{B}}_{K}
\]
 that contains all sets
\( \BCE\in \mathcal{B}^{K} \)  of the form \( \BCE=\CE_{1}\times \ldots \times \CE_{K-1}\times E \)
with \( \CE_{k}\in \mathcal{B}, \, k=1,\ldots,K-1. \) Here we take into account  the fact that the stopping set \( \CE_{K} \)
must coincide with \( E. \) Let \( \mathfrak{S} \) be a subset of \( \mathfrak{B}. \) Define
\begin{eqnarray*}
\BCE_{M}:=\arg\max_{\BCE\in
\mathfrak{S}}\left\{ \frac{1}{M}\sum_{m=1}^{M}G_{\tau_{1}(\BCE)}
\left(X^{(m)}_{\tau_{1}(\BCE)}\right) \right\}.
\end{eqnarray*}
The stopping rule
\begin{eqnarray*}
    \tau_{M}:=\tau_{1}(\BCE_M)=\min\{ 1\leq k \leq K: X_{k}\in \mathcal{S}_{M,k} \}.
\end{eqnarray*}
is generally suboptimal
and therefore the corresponding Monte Carlo estimate
\begin{eqnarray}
\label{VMN}
V_{M,N}:=\frac{1}{N}\sum_{n=1}^{N}G_{\tau^{(n)}_{M}}
\left(\widetilde X^{(n)}_{\tau^{(n)}_{M}}\right)
\end{eqnarray}
with
\begin{eqnarray*}
    \tau^{(n)}_{M}:=\min\{ 1\leq k \leq K: \widetilde X^{(n)}_{k}\in \mathcal{S}_{M,k} \}, \quad n=1,\ldots,N
\end{eqnarray*}
based on a new, independent of \( ( X^{(1)},\ldots, X^{(M)}) \) set of
trajectories
\[
 (\widetilde X^{(n)}_{0},\ldots,\widetilde X^{(n)}_{K}), \quad n=1,\ldots,N,
\]
fulfills
\begin{eqnarray}
\label{VMN}
V_M:=\E_{x} \left[V_{M,N}|X^{(1)},\ldots,X^{(M)}\right]
&\leq&\sup_{\BCE\in
\mathfrak{S}}\E_{x}\left[G_{\tau_{1}(\BCE)}\left(X_{\tau_{1}(\BCE)}\right)\right].
\end{eqnarray}
If the set \( \mathfrak{S} \) is rich enough, then
\begin{eqnarray*}
    \sup_{\BCE\in
\mathfrak{S}}\E_{x}\left[G_{\tau_{1}(\BCE)}\left(X_{\tau_{1}(\BCE)}\right)\right]=:\E_{x}\left[G_{\tau_{1}(\bar\BCE)}
\left(X_{\tau_{1}(\bar\BCE)}\right)\right]\approx\E_{x}\left[G_{\tau_{1}(\BCE^{*})}\left(X_{\tau_{1}(\BCE^{*})}\right)\right]
\end{eqnarray*}
and \( V_{M,N} \) can serve as a good approximation for \( V^{*} \) for large enough \( M \) and \( N. \)
In the next section we are going to study the question: how fast does \( V_{M} \) converge to \( V^{*}=V^{*}_{1} \) as \( M\to \infty \) ? We will show that the corresponding rates of convergence are always faster than usual rates \( M^{-1/2}. \)  This fact has a practical implication since it indicates that \( M \), the number of simulated paths used in the optimization step, can be taken much smaller than \( N, \) the number of paths used to compute the final estimate \( V_{M,N} \).

\section{Main results}

\paragraph{Definition}
Let \( \delta>0 \) be a given number and \( d_{X}(\cdot,\cdot) \)
be a pseudedistance
between two elements of \( \mathfrak{B} \) defined as
\begin{eqnarray}
\label{DX}
&&d_{X}(G_{1}\times\ldots\times G_{K},G'_{1}\times\ldots\times G'_{K})=\sum_{k=1}^{K}\P_{x}(X(t_{k})\in G_{k}\triangle G'_{k}),
\end{eqnarray}
where \( \{ G_{k} \}\) and \( \{ G'_{k} \} \) are subsets of \( E. \)
Define \( N(\delta,\mathfrak{S},d_{X}) \)
be the smallest value  \(n \) for which there exist pairs of sets
\[
(G_{j,1}^{L}\times\ldots\times G_{j,K}^{L},G_{j,1}^{U}\times\ldots\times G_{j,K}^{U} ), \quad j=1,\ldots,n,
\]
such that \( d_{X}(G_{j,1}^{L}\times\ldots\times G_{j,K}^{L},G_{j,1}^{U}\times\ldots\times G_{j,K}^{U} ) \leq \delta\) for all
\( j=1,\ldots,n, \) and for any \( G\in \mathfrak{S} \) there exists
\( j(G)\in \{ 1,\ldots,n \} \) for which
\[
G^{L}_{j(G),k}\subseteq G_{k} \subseteq
G^{U}_{j(G),k}, \quad k=1,\ldots,K.
\]
Then the value \( \mathcal{H}(\delta,\mathfrak{S},d):=\log[N(\delta,\mathfrak{S},d_{X})] \)
is called the \( \delta \)-entropy with bracketing of \( \mathfrak{S}\) for
the pseudedistance \( d_{X} \).
\paragraph{Assumption}
We assume that the family of stopping regions \( \mathfrak{S} \) is such that
\begin{eqnarray}
\label{CA}
\mathcal{H}(\delta,\mathfrak{S},d_{X})\leq A\delta^{-\rho}
\end{eqnarray}
for some constant \( A>0 \), any \( 0<\delta<1 \) and some \( \rho>0 \).
\paragraph{Example} Let \( \mathfrak{S}=\mathfrak{S}_{\gamma} \), where \( \mathfrak{S}_{\gamma} \) is a class of
subsets of \( \overbrace{\mathbb{R}^{d}\times\ldots\times\mathbb{R}^{d}}^K \) with boundaries of H\"older smoothness \( \gamma>0 \)
defined as follows. For given \( \gamma>0 \) and \( d\geq 2 \) consider the functions
\( b(x_{1},\ldots,x_{d-1}), \) \( b:\mathbb{R}^{d-1}\to \mathbb{R} \) having continuous partial
derivatives of order \( l \), where \( l \) is the maximal integer that is strictly
less than \( \gamma \). For such functions \( b \), we denote the Taylor polynomial
of order \( l \) at a point  \( x\in \mathbb{R}^{d-1} \) by \( \pi_{b,x} \). For a given
\( H>0 \), let \( \Sigma(\gamma,H) \) be the class of functions  \( b \) such that
\begin{eqnarray*}
|b(y)-\pi_{b,x}(y)|\leq H\|x-y\|^{\gamma}, \quad x,y\in \mathbb{R}^{d-1},
\end{eqnarray*}
where \( \| y \| \) stands for the Euclidean norm of \( y\in \mathbb{R}^{d-1}. \)
Any  function \( b \) from \( \Sigma(\gamma,H) \) determines a set
\begin{eqnarray*}
S_{b}:=\{ (x_{1},\ldots,x_{d})\in \mathbb{R}^{d}: 0\leq x_{d}\leq b(x_{1},\ldots,x_{d-1})
\}.
\end{eqnarray*}
Define the class
\begin{eqnarray}
\label{BFC}
\mathfrak{S}_{\gamma}:=\{ S_{b_{1}}\times\ldots\times S_{b_{K-1}}\times E :
\, b_{1},\ldots,b_{K-1}\in \Sigma(\gamma,H) \}.
\end{eqnarray}
It can be shown \cite[see][Section 8.2]{D}  that the class \( \mathfrak{S}_{\gamma} \) fulfills
\[
\mathcal{H}(\delta,\mathfrak{S}_{\gamma},d_{X})\leq A\delta^{-(K-1)(d-1)/\gamma}
\]
for some \( A>0 \)  and all \( \delta>0\) small enough. Now we are in the position to formulate the main result of our study.
\begin{thm}
\label{UP}
Let \( \mathfrak{S} \) be a subset of  \( \mathfrak{B} \) such that assumption \eqref{CA} is fulfilled
with some \( 0<\rho\leq 1 \) and
\begin{eqnarray}
    \label{AppQ}
    &&\E_{x}\left[G_{\tau_{1}(\BCE^{*})}\left(X_{\tau_{1}(\BCE^{*})}\right)\right]-\bar V \leq D M^{-1/(1+\rho)}
\end{eqnarray}
with \( \bar V:= \sup_{\BCE\in
\mathfrak{S}}\E_{x}\left[G_{\tau_{1}(\BCE)}\left(X_{\tau_{1}(\BCE)}\right)\right] \) and some constant \( D>0. \)
Assume that  all functions \( G_{k} \) are uniformly bounded and the inequalities
\begin{eqnarray}
\label{BA}
    &&\P_{x}(|G_{k}(X_{k})-\E[V^{*}_{k+1}(X_{k+1})|\mathcal{F}_{k}]|<\delta)\leq A_{0,k}\delta^{\alpha}, \quad \delta<\delta_{0}
\end{eqnarray}
hold for  some \( \alpha>0 \), \( A_{0,k}>0, \) \( k=1,\ldots, K-1 \), and  \( \delta_{0}>0 \). Then for any \( U>U_{0} \)
and \( M>M_{0} \)
\begin{eqnarray}
\label{UBounds}
    \P^{\otimes M}_{x}\left(V^{*}-V_{M}\geq (U/M)^{\frac{1+\alpha }{2+\alpha(1+\rho)}}\right)\leq C\exp(-\sqrt{U}/B).
\end{eqnarray}
 with some constants \( U_{0}>0 \), \( M_{0}>0 \), \( B>0 \) and \( C>0. \)
\end{thm}
\begin{rem}
\label{NApp}
Without condition \eqref{AppQ} the inequality \eqref{UBounds} continues to hold with \( V^{*} \) replaced by
\( \bar V \), the best approximation
of \( V^{*} \) within the class of stopping regions \( \mathfrak{S}. \)
\end{rem}
\begin{rem}
The requirement that functions \( G_{k} \) are uniformly bounded can be replaced by the existence
of all moments of \( G_{k}(X_{k}),\, k=1,\ldots,K-1, \) under \( \P. \)  In this case on can reformulate Theorem~\ref{EINEQ} using
generalized entropy with bracketing instead of usual entropy with bracketing (see Chapter 5.4 in \citet{V}).
\end{rem}
The above convergence rates can not be in general improved as shown in the next theorem.
\begin{prop}
\label{LB}
Consider the problem \eqref{OSP} with \( k=1 \) and two possible stopping dates, i.e. \( \tau \in \{ 1,2 \} \). Fix a pair of non-zero functions \( G_{1},G_{2} \) such that
\( G_{2}: \mathbb{R}^{d}\to \{ 0,1 \} \) and \( 0<G_{1}(x)<1 \) on \( [0,1]^{d}. \)
Fix some \( \gamma>0 \) and \( \alpha>0 \) and let \( \mathcal{P}_{\alpha,\gamma } \) be a class of pricing  measures
such that the
condition \eqref{BA} is fulfilled  and for any \( \P\in \mathcal{P}_{\alpha,\gamma } \) the corresponding stopping set \( \BCE^{*}_{\P}\) is in \(\mathfrak{S}_{\gamma}. \)
Then there exist a subset \( \mathcal{P}\) of \( \mathcal{P}_{\alpha, \gamma } \)
and a constant \( B>0 \) such that for any \( M\geq 1 \), any
stopping time
\( \tau_{M} \in \{ 1,2 \}\)
measurable w.r.t. \( \mathcal{F}^{\otimes M} \)
\begin{eqnarray*}
 \sup_{\P\in \mathcal{P}}\left\{ \sup_{\tau\in\{ 1,2 \}}\E_{\P}[G_{\tau}(X_\tau)]-\E_{\P^{\otimes M}}[\E_{\P}G_{\tau_{M}}(X_{\tau_{M}})] \right\}\geq B
 M^{-\frac{1+\alpha }{2+\alpha(1+(d-1)/\gamma )}}.
\end{eqnarray*}
\end{prop}

\paragraph{Discussion} It follows from Theorem~\ref{UP} that
\begin{eqnarray*}
    V^{*}- V_{M} =O_{\P}\left(M^{-\frac{1+\alpha }{2+\alpha(1+\rho)}}\right)=o_{\P}(M^{-1/2})
\end{eqnarray*}
as long as \( \alpha >0. \)  Using the decomposition
\begin{eqnarray*}
    V^{*}- V_{M,N} =V^{*}-V_{M} +
     V_{M} -V_{M,N }
\end{eqnarray*}
and the fact that \( V_{M} -V_{M,N }=O_{\P}(1/\sqrt{N}) \) for any \( M>0 \), we conclude that
\begin{eqnarray*}
  V^{*}- V_{M,N}= O_{\P}\left(M^{-\frac{1+\alpha }{2+\alpha(1+\rho)}}+N^{-\frac{1}{2}}\right).
\end{eqnarray*}
Hence, given \( N \), a reasonable choice
of \( M \), the number of Monte Carlo paths used in the optimization step, can be defined as  \( M\asymp N^{\frac{2+\alpha(1+\rho)}{2(1+\alpha)}}.\) In the case when there exists a parametric family of stopping regions
satisfying \eqref{AppQ} (see Section~\ref{app_sec} for some examples), one gets
\begin{eqnarray}
\label{MPar}
M\asymp N^{\frac{2+\alpha}{2(1+\alpha)}}
\end{eqnarray}
since any parametric family of
stopping regions with finite dimensional parameter set  fulfills \eqref{CA} for arbitrary small \( \rho>0.  \)
Let us also make a few remarks on the condition \eqref{BA} and the parameter \( \alpha \). If each function
\(G_{k}(x)-\E_{k,x}[V^{*}_{k+1}(X_{k+1})], \, k=1,\ldots, K-1, \) has a non-vanishing Jacobian in the vicinity of the stopping
boundary \( \partial \CE_{k} \) and \( X_{k} \) has continuous distribution, then  \eqref{BA} is fulfilled with \( \alpha=1. \) In fact, it is not difficult to construct examples showing that the parameter \( \alpha \) can take any value from \( \mathbb{R}_{+} \). If \( \alpha=1 \) (the most common case) \eqref{MPar} simplifies to \( M\asymp N^{3/4} \), the choice supported by our numerical example.
\par
Finally,  we would like to mention an interesting methodological connection between our analysis and  the
analysis of statistical discrimination problem performed in  \citet{MT} (see also \citet{DGL}). In particular, we need similar results form the theory of empirical processes and the condition \eqref{BA} formally resembles the so called ``margin''
condition often encountered in the literature on discrimination analysis.

\section{Applications}
\label{app_sec}
In this section we illustrate our theoretical results by some financial applications.
Namely, we consider the problem of pricing Bermudan options.
The pricing of American-style options is one of the most challenging problems in
computational finance, particularly when more than one factor affects the option values.
Simulation based methods have become increasingly attractive
 compared to
other numerical methods as the dimension of the problem increases. The reason for this is
that the convergence rates of simulation based methods are generally independent of the number
of state variables.
In the context of our paper we consider the so called parametric approximation algorithms
\cite[see][Section 8.2]{Gl}.
In essence, these algorithms represent the optimal stopping sets \( \mathcal{S}^{*}_{k} \) by a finite numbers of parameters and
then find the Bermudan option price by maximizing, over the parameter space, a Monte Carlo approximation of the corresponding value function.
The important question here is wether on can parametrize the optimal stopping region \( \BCE^{*} \) by a finite dimensional set of parameters, i.e. \( \BCE^{*}=\BCE(\theta), \, \theta \in \Theta, \) where \( \Theta  \) is a compact finite dimensional set. It turns out that that
this is possible in many situations (see \citet{G}). The assumption \eqref{CA} and \eqref{AppQ} are then automatically fulfilled with arbitrary small  \( \rho>0. \)
\subsection{Numerical example: Bermudan max call}
This is a benchmark example studied in \citet{BG} and \citet{Gl}
among others. Specifically, the model with $d$ identically distributed
assets is considered, where each underlying has dividend yield $\delta $.
The risk-neutral dynamic of the asset \( X(t)=(X^{1}(t),\ldots,X^{d}(t)) \) is given by
\begin{equation*}
\frac{dX^{l}(t)}{X^{l}(t)}=(r-\delta )dt+\sigma dW^{l}(t),\quad X^{l}(0)=x_{0}, \quad l=1,...,d,
\end{equation*}%
where $W^{l}(t),\,l=1,...,d$, are independent one-dimensional Brownian
motions and $x_{0}, r,\delta ,\sigma $ are constants. At any time $t\in
\{t_{1},...,t_{K}\}$ the holder of the option may exercise it and
receive the payoff
\begin{equation*}
G_{k}(X_{k}):=\left(\max \left(X^{1}_{k},...,X^{d}_{k}\right)-\kappa\right)^{+},
\end{equation*}%
where \( X_{k}:=X(t_{k}) \) for \( k=1, \ldots, K. \)
We take \( d=2 \), \( r=5\% \), \( \delta=10\% \), \( \sigma=0.2 \), \( \kappa=100 \), \( x_{0}=90 \) and $t_{k}=kT/K,\,k=1,\ldots, K$, with $T=3,\,K
=9$ as in \citet[Chapter 8]{Gl}.
\par
To describe the optimal early exercise region at date \( t_{k},\, k=1,\ldots, K, \) one can divide \( \mathbb{R}^{2} \) into three different connected sets: one exercise region and two continuation regions (see \citet{BD} for more details).
All these regions can be parameterized by using two functions depending on two dimensional parameter
\( \theta_{k}\in \mathbb{R}^{2}. \)  Making use of this characterization, we define a parametric family of stopping regions as in \citet{G} via
\begin{eqnarray*}
    \CE_{k}(\theta_{k}):=\{(x_{1},x_{2}): \max(\max(x_{1},x_{2})-K,0)>\theta^{1}_{k};\, |x_{1}-x_{2}|>\theta^{2}_{k} \},
\end{eqnarray*}
where \( \theta_{k}\in \Theta, \,k=1,\ldots, K\) and \( \Theta  \) is a compact subset of \( \mathbb{R}^{2}. \)
Furthermore, we simplify the corresponding optimization problem by setting \( \theta_{1}=\ldots=\theta_{K}.\)
This will introduce an additional bias and hence may increase the left hand side of \eqref{AppQ} (see  Remark~\ref{NApp}). However, this bias turns out to be rather small in practice.
In order to implement and analyze the simulation based optimization based algorithm in this situation,  we  perform the following steps:
\begin{itemize}
  \item Simulate \( L \) independent sets of trajectories of the process \( (X_{k}) \) each of the size \( M \):
  \begin{eqnarray*}
    (X^{(l,m)}_{1},\ldots,X^{(l,m)}_{K}),\quad m=1,\ldots,M,
  \end{eqnarray*}
  where \( l=1,\ldots, L. \)
  \item Compute estimates \( \theta_{M}^{(1)},\ldots, \theta_{M}^{(L)} \) via
\begin{eqnarray*}
 \theta_{M}^{(l)}:=\arg\max_{\theta\in
\Theta}\left\{ \frac{1}{M}\sum_{m=1}^{M}G_{\tau_{1}(\BCE(\theta))}
\left(X^{(l,m)}_{\tau_{1}(\BCE(\theta))}\right) \right\}.
\end{eqnarray*}
\item Simulate a new set of trajectories of size \( N \) independent
of \( (X^{(l,m)}_{k}): \)
\[
(\widetilde X^{(n)}_{1},\ldots, \widetilde X^{(n)}_{K}), \quad  n=1,\ldots,N.
\]
\item Compute \( L \) estimates for the optimal value function \( V^{*}_{1} \) as follows
\begin{eqnarray*}
V^{(l)}_{M,N}:=\frac{1}{N}\sum_{n=1}^{N}G_{\tau^{(l,n)}_{M}}
\left(\widetilde X^{(n)}_{\tau^{(l,n)}_{M}}\right), \quad l=1,\ldots, L,
\end{eqnarray*}
with
\begin{eqnarray*}
    \tau^{(l,n)}_{M}:=\min\left\{ 1\leq k \leq K: \widetilde X^{(n)}_{k}\in \mathcal{S}_{k}\left(\theta^{(l)}_{M}\right) \right\}, \quad n=1,\ldots,N.
\end{eqnarray*}
 Denote by \( \sigma_{M,N,l}  \) the standard deviation computed from the sample \( (G_{\tau^{(l,n)}_{M}}, \, n=1,\ldots, N)\)
 and set \( \sigma_{M,N}=\min_{l} \sigma_{M,N,l}.  \)
 \item Compute
\begin{eqnarray*}
    \mu_{M,N,L}:=\frac{1}{L}\sum_{l=1}^{L}V^{(l)}_{M,N}, \quad \vartheta_{M,N,L}:=\sqrt{\frac{1}{L-1}\sum_{l=1}^{L}\left(V^{(l)}_{M,N}-\mu _{M,N,L}\right)^{2}}.
\end{eqnarray*}
\end{itemize}
By the law of large numbers
\begin{eqnarray}
\label{mu_conv}
\mu_{M,N,L}&\stackrel{\P}{\to}& \E_{\P^{\otimes M}} \left[ V_{M,N}\right],\quad L\to \infty,
\\
\label{vartheta_conv}
\vartheta_{M,N,L}&\stackrel{\P}{\to}& \Var_{\P^{\otimes M}} \left[V_{M,N} \right], \quad L\to \infty,
\end{eqnarray}
where
\begin{eqnarray*}
     V_{M,N}:= \frac{1}{N}\sum_{n=1}^{N}G_{\tau^{(n)}_{M}}
\left(\widetilde X^{(n)}_{\tau^{(n)}_{M}}\right).
\end{eqnarray*}
The difference
\( \bar V- V_{M,N} \)
with
\begin{eqnarray*}
\bar V:=\max _{\theta\in \Theta}\E [G_{\tau_{1}(\BCE(\theta))}(X_{\tau_{1}(\BCE(\theta))})]
\end{eqnarray*}
can be decomposed into the sum of three terms
\begin{eqnarray}
   \label{decomp_NE}
   (\bar V-\E_{\P^{\otimes M}}  \left[ V_{M} \right])+ (\E_{\P^{\otimes M}}  \left[ V_{M} \right]-V_{M})+ V_{M}-V_{M,N}.
   \end{eqnarray}
 The first term in \eqref{decomp_NE} is deterministic and can be approximated by
 \( Q_{1}(M):=\mu_{M^{*},N^{*},L^{*}} - \mu_{M,N^{*},L^{*}}\) with large enough \( L^{*} \), \( M^{*} \) and \( N^{*}. \) The variability of the second, zero mean, stochastic term   can be measured by   \(\sqrt{\Var_{\P^{\otimes M}} \left[V_{M} \right]} \) which in turn can be estimated by \(  Q_{2}(M):=\sqrt{\vartheta_{M,N^{*},L^{*}}} \), due to \eqref{vartheta_conv}.   The standard deviation of  \(V_{M}-V_{M,N}\) for any \( M \) can be approximated by \( Q_{3}(N)=\sigma_{M^{*},N} /\sqrt{N} \).
In our simulation study we take \( N^{*}=1000000, \, L^{*}=500,\, M^{*}=10000 \)  and obtain \( \bar V\approx \mu_{M^{*},N^{*},L^{*}}=7.96 \) (note that \(  V^{*}=8.07 \)  according to
\citet{Gl}). In the left-hand side of Figure~\ref{Relation_MN} we plot both quantities \(  Q_{1}(M) \) and
\( Q_{2}(M) \) as functions of \( M. \) Note that \( Q_{2}(M) \) dominates \(  Q_{1}(M) \), especially for large \( M. \)
Hence, by comparing \( Q_{2}(M) \) with \( Q_{3}(N) \) and approximately solving the equation \( Q_{2}(M)=Q_{3}(N) \) in \( N \), one can infer on the optimal relation between \( M \) and \( N \).
In Figure~\ref{Relation_MN} (on the right-hand side) the resulting empirical relation is depicted by crosses. Additionally,  we plotted two benchmark curves \( N=M^{4/3} \) and \( N=M^{4.5/3} \).
As one can see the choice \( M=N^{3/4} \) is likely to be sufficient in this situation since it always leads to the inequality \( Q_{1}(M)+\sigma Q_{2}(M)\leq  \sigma Q_{3}(N)  \) for any \( \sigma>1. \) As a consequence, for  \( M=N^{3/4} \) and any \( N \), \( \bar V \)
lies with high probability in the interval \( [\mu_{M,N,L^{*}}-\sigma Q_{3}(N),\mu_{M,N,L^{*}}+\sigma Q_{3}(N)],\) provided that \( \sigma \) is large enough.
\begin{figure}[pth]
\centering \includegraphics[width=14cm]{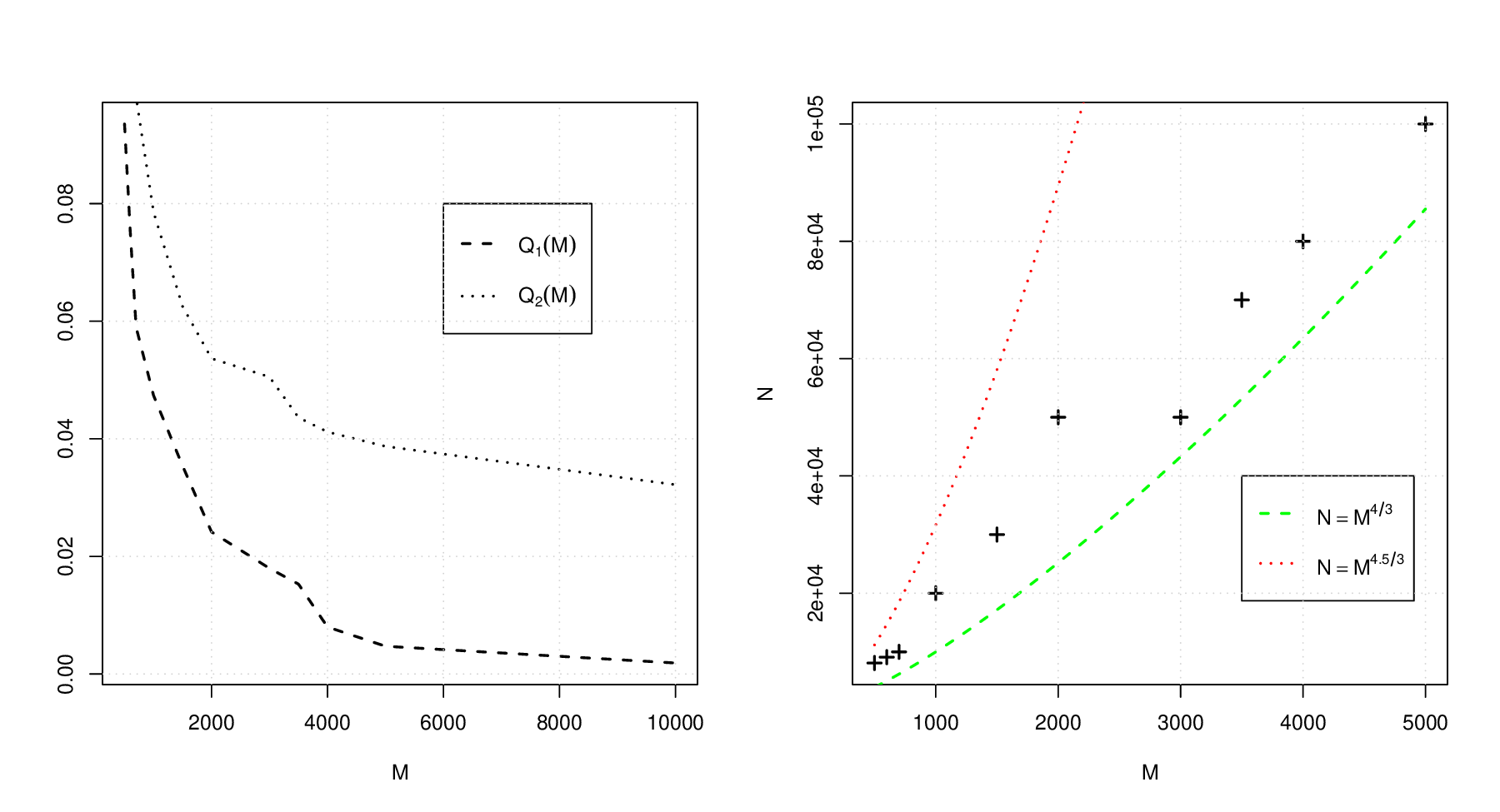}
\caption{ Left: functions \( Q_{1}(M) \) and \( Q_{2}(M) \); Right: optimal empirical relationship between \( M \) and \( N \) (crosses) together with benchmark curves \( N=M^{4/3} \) (dashed line) and \( N=M^{4.5/3} \) (dotted line). }%
\label{Relation_MN}%
\end{figure}

\section{Proof of main results}
\subsection{Proof of Theorem~\ref{UP}}
Define
\begin{eqnarray*}
    \Delta_{M}(\BCE)&:=&\sqrt{M}\sum_{m=1}^{M}\left\{ G_{\tau_{1}(\BCE)}\left(X^{(m)}_{\tau_{1}(\BCE)}\right)
    -\E \left[ G_{\tau_{1}(\BCE)}\left(X_{\tau_{1}(\BCE)}\right) \right] \right\}
\end{eqnarray*}
and \( \Delta_{M}(\BCE',\BCE):=\Delta_{M}(\BCE')-\Delta_{M}(\BCE) \) for any \( \BCE',\BCE\in \mathfrak{S}. \)
Since
\[
\frac{1}{M}\sum_{m=1}^{M}G_{\tau_{1}(\bar\BCE)}\left(X^{(m)}_{\tau_{1}(\bar\BCE)}\right) \leq
\frac{1}{M}\sum_{m=1}^{M}G_{\tau_{1}(\BCE_{M})}\left(X^{(m)}_{\tau_{1}(\BCE_{M})}\right)
\]
with probability \( 1 \), it holds
\begin{eqnarray}
\label{PUB1}
\Delta(\BCE_{M})&\leq &\Delta(\bar \BCE)+\frac{\left[ \Delta_{M}(\BCE^{*},\bar \BCE)
+\Delta_{M}(\BCE_{M},\BCE^{*}) \right]}{\sqrt{M}}
\end{eqnarray}
with \( \Delta(\BCE):=\E[G_{\tau^{*}_{1}}(X_{\tau^{*}_{1}})]-\E[G_{\tau_{1}(\BCE)}(X_{\tau_{1}(\BCE)})]. \)
Set  \( \varepsilon_{M}=M^{-1/2(1+\rho)} \) then
\begin{eqnarray*}
\Delta(\BCE_{M})&\leq& \Delta(\bar \BCE)+
\frac{2}{\sqrt{M}}\sup_{\BCE\in \mathfrak{S}:\, \Delta_{G}(\BCE^{*},\BCE)\leq \varepsilon_{M} }
|\Delta_{M}(\BCE^{*},\BCE)|
\\
&&+2\times\frac{\Delta_{G}^{(1-\rho)}(\BCE^{*},\BCE_{M})}{\sqrt{M}}\times
\sup_{\BCE\in \mathfrak{S}:\, \Delta_{G}(\BCE^{*},\BCE)>\varepsilon_{M} }
\left[ \frac{|\Delta_{M}(\BCE^{*},\BCE)|}
{\Delta_{G}^{(1-\rho)}(\BCE^{*},\BCE)} \right].
\end{eqnarray*}
Define
\begin{eqnarray*}
\mathcal{W}_{1,M}:=\sup_{\BCE\in \mathfrak{S}:\, \Delta_{G}(\BCE^{*},\BCE)\leq \varepsilon_{M} }
|\Delta_{M}(\BCE^{*},\BCE)|,
\\
\mathcal{W}_{2,M}:=\sup_{\BCE\in \mathfrak{S}:\, \Delta_{G}(\BCE^{*},\BCE)>\varepsilon_{M} }
\frac{|\Delta_{M}(\BCE^{*},\BCE)|}
{\Delta_{G}^{(1-\rho)}(\BCE^{*},\BCE)}
\end{eqnarray*}
and set \( \mathcal{A}_{0}:=\{ \mathcal{W}_{1,M}\leq U\varepsilon_{M} ^{1-\rho}  \} \)
for \( U>U_{0}. \)
Note that under assumption \eqref{CA} the condition \eqref{CA1} of Theorem~\ref{EINEQ} is fulfilled with \( \nu = 2\rho  \)
due to Corollary~\ref{DFX}. Hence Theorem~\ref{EINEQ} yields \( \P(\bar{\mathcal{A}}_{0})\leq  C\exp(-U\varepsilon_{M} ^{-2\rho }/C^{2} ).\)
Denote
\begin{eqnarray*}
\Delta_{G}(\BCE,\BCE'):=\left\{ \E\left[G_{\tau_{1}(\BCE)}
\left(X_{\tau_{1}(\BCE)}\right)-G_{\tau_{1}(\BCE')}
\left(X_{\tau_{1}(\BCE')}\right)\right]^{2} \right\}^{1/2}
\end{eqnarray*}
for any \( \BCE,\BCE' \in \mathfrak{B}.\)
Since \( \Delta(\bar \BCE)\leq D M^{-1/(1+\rho )}   \) and \( \varepsilon^{1-\rho}_{M}/\sqrt{M}=M^{-1/(1+\rho)} \),   we get on \( \mathcal{A}_{0} \)
\begin{eqnarray*}
\Delta(\BCE_{M})&\leq&  C_{0}M^{-1/(1+\rho )}+2\times\frac{\Delta_{G}^{(1-\rho)}(\BCE^{*},\BCE_{M})}{\sqrt{M}}\mathcal{W}_{2,M}
\end{eqnarray*}
with  \( C_{0}=D+2U \).
Combining  Corollary~\ref{DFX} with Corollary~\ref{DDX} leads to the inequality
\begin{eqnarray*}
\Delta_{G}(\BCE^{*},\BCE_{M})\leq 2\sqrt{2}A_{G}v^{-\alpha/2(1+\alpha )}_{\alpha }\Delta^{\alpha/2(1+\alpha)}(\BCE_{M})
\end{eqnarray*}
which holds on the set \( \mathcal{A}_{1}:=\{ \Delta_{X}(\BCE^{*},\BCE_{M})\leq \delta_{\alpha } \}, \)
where \( \delta_{\alpha }  \) and \( v_{\alpha } \) are defined in Corollary~\ref{DDX}.
Denote
\begin{eqnarray*}
\mathcal{A}_{2}:=\left\{ \Delta(\BCE_{M})>C_{0}(1+\varkappa )M^{-1/(1+\rho)} \right\}
\end{eqnarray*}
with some \( \varkappa  >0. \) It then holds on \( \mathcal{A}_{0}\cap\mathcal{A}_{1}\cap\mathcal{A}_{2} \)
\begin{eqnarray*}
    \Delta(\BCE_{M})\leq 2\frac{ \Delta^{\alpha(1-\rho)/(2(1+\alpha))}(\BCE_{M})}{\varkappa \sqrt{M}}\mathcal{W}_{2,M}
\end{eqnarray*}
and therefore
\begin{eqnarray*}
    \Delta(\BCE_{M})\leq (\varkappa/2)^{-\nu} M^{-\nu/2}\mathcal{W}^{\nu}_{2,M}
\end{eqnarray*}
with \( \nu=\frac{2(1+\alpha) }{2+\alpha (1+\rho )}. \)
Let us now estimate \( \P(\bar{\mathcal{A}}_{1}). \)
Using Corollary~\ref{DDX}, we get
\begin{multline*}
    \P_{x}^{\otimes M}(\Delta_{X}(\BCE^{*},\BCE_{M})>\delta_{\alpha } )\leq
    \\
    \P_{x}^{\otimes M}\left( \left( \frac{2^{1/\alpha }}{\delta_{0}} \right)\Delta(\BCE_{M})+\frac{\delta_{\alpha}}{2(1+\alpha)}>\delta_{\alpha }  \right)
    \\
    =\P_{x}^{\otimes M}(\Delta(\BCE_{M})>c_{\alpha})
\end{multline*}
with \( c_{\alpha}=\delta_{0}\delta_{\alpha}2^{-1/\alpha}\left( 1-\frac{1}{2(1+\alpha )} \right).  \)
Furthermore, due to \eqref{PUB1}
\begin{eqnarray*}
     \P_{x}^{\otimes M}(\Delta(\BCE_{M})>c_{\alpha})&\leq & \P_{x}^{\otimes M}\left(DM^{-1/(1+\rho)}+2M^{-1/2}\sup_{\BCE\in \mathfrak{S}}|\Delta_{M}(\BCE)|>c_{\alpha}\right)
     \\
     &\leq &
      \P^{\otimes M}_{x}\left(\sup_{\BCE\in \mathfrak{S}}|\Delta_{M}(\BCE)|>c_{\alpha}\sqrt{M}/4\right)
\end{eqnarray*}
for large enough \( M. \) Theorem~\ref{EINEQ} implies
\begin{eqnarray*}
     \P^{\otimes M}_{x}\left(\sup_{\BCE\in \mathfrak{S}}|\Delta_{M}(\BCE)|>c_{\alpha}\sqrt{M}/4\right)\leq B_{1}\exp(-MB_{2})
\end{eqnarray*}
with some constants \( B_{1}>0 \) and  \( B_{2}=B_{2}(\alpha)>0. \)
Applying Theorem~\ref{EINEQ} to \( \mathcal{W}^{\nu}_{2,M} \) and using the fact that \( \nu/2\leq 1/(1+\rho)  \) for all \( 0<\rho\leq 1, \)
 we  finally obtain the inequality
\begin{eqnarray*}
    \P^{\otimes M}_{x}\left(\Delta(\BCE_{M})>(V/M)^{\nu/2} \right)&\leq & C\exp(-\sqrt{V}/B_{3})
    \\
    &&+C\exp\left( -\frac{U\varepsilon_{M}^{-2\rho}}{C^{2}} \right)+B_{1}\exp(-MB_{2})
\end{eqnarray*}
which holds for all \( V>V_{0} \) and \( M>M_{0} \) with some constant \( B_{3} \) depending on \( \varkappa.  \)

\subsection{Proof of Proposition~\ref{LB}}
For simplicity, we give the proof only for the case $d=2$ (an extension to
higher dimensions is straightforward). In the case of  two exercise dates  the corresponding optimal stopping problem is completely specified by the
distribution of the vector $(X_{1},G_{2}(X_{2})).$ Because of a digital
structure of $G_{2}$ the distribution of $(X_{1},G_{2}(X_{2}))$ would be
completely determined if the marginal distribution of $X_{1}$ and the
probability $\P(G_{2}(X_{2})=1|X_{1}=x)$ are defined. Taking into account
this, we now construct a family of distributions for $%
(X_{1},G_{2}(X_{2}))$ indexed by elements of the set $\Omega =\{0,1\}^{m}.$
First, the marginal distribution of $X_{1}$ is supposed to be the same for
all $\omega \in \Omega $ and posseses a density $p(x)$ satisfying
\begin{equation*}
0<p_{\ast }\leq p(x)\leq p^{\ast }<\infty ,\quad x\in \lbrack 0,1]^{2}.
\end{equation*}%
Let us now construct a family of conditional distributions $\P_{\omega
}(G_{2}(X_{2})=1|X_{1}=x)$, $\omega \in \Omega .$ To this end let $%
\phi $ be an infinitely many times differentiable function on $\mathbb{R}$
with the following properties: $\phi (z)=0$ for $|z|\geq 1,$ $\phi (z)\geq 0$
for all $z$ and $\sup_{z\in \mathbb{R}}[\phi (z)]\leq 1.$ For $j=1,\ldots ,m$
put%
\begin{equation*}
\phi _{j}(z):=\delta m^{-\gamma }\phi \left( m\left[ z-\frac{2j-1}{m}\right]
\right),  \quad z\in \mathbb{R}
\end{equation*}%
with some $0<\delta <1.$ For vectors $\omega =(\omega _{1},\ldots ,\omega
_{m})$ of elements $\omega _{j}\in \{0,1\}$ and for any $z\in \mathbb{R}$
define%
\begin{equation*}
b(z,\omega ):=\sum_{j=1}^{m}\omega _{j}\phi _{j}(z).
\end{equation*}%
Put  for any $\omega \in \Omega $ and any $x\in \mathbb{R}^{2},$%
\begin{eqnarray*}
C_{\omega }(x) &:=&\P_{\omega }(G_{2}(X_{2})=1|X_{1}=x)= \\
&=&G_{1}(x)-Am^{-\gamma /\alpha }\mathbf{1}\left\{ 0\leq x_{2}\leq
b(x_{1},\omega )\right\}  \\
&&+Am^{-\gamma /\alpha }\mathbf{1}\left\{ b(x_{1},\omega )<x_{2}\leq \delta
m^{-\gamma }\right\} ,
\end{eqnarray*}%
where $A$ is a positive constant. Due to our assumptions on $G_{1}(x)$, there
are constants $0<G_{-}<G_{+}<1$ such that
\begin{equation*}
G_{-}\leq G_{1}(x)\leq G_{+}, \quad x\in [0,1]^{2}.
\end{equation*}%
Hence, the constant $A$ can be chosen in such a way that  $C_{\omega }(x)$
remains positive and strictly less than $1$ on $[0,1]^{2}$ for any \( \omega\in \Omega. \)  The
stopping set%
\begin{equation*}
\mathcal{S}_{\omega }:=\left\{ x:C_{\omega }(x)\leq G_{1}(x)\right\} =\left\{
(x_{1},x_{2}):0\leq x_{2}\leq b(x_{1},\omega )\right\}
\end{equation*}%
belongs to $\mathfrak{S}_{\gamma}$ since $b(\cdot ,\omega )\in \Sigma (\gamma ,L)$
for $\delta$ small enough. Moreover, for any $\eta>0 $
\begin{eqnarray*}
\P_{\omega }\left( |G_{1}(X_{1})-C_{\omega }(X_{1})|\leq \eta \right)
&=&\P_{\omega }( 0\leq X_{1}^{2}\leq
\delta
m^{-\gamma })\mathbf{1}(Am^{-\gamma /\alpha }\leq \eta)
\\
&\leq &\delta p^{\ast }m^{-\gamma }\mathbf{1}(Am^{-\gamma /\alpha }\leq \eta)\leq \delta p^{\ast
}A^{-\alpha }\eta ^{\alpha }
\end{eqnarray*}
and the condition \eqref{BA} is fulfilled.
Let $\tau _{M}$ be a stopping time w.r.t. $\mathcal{F}^{\otimes M}$, then
the identity (see Lemma~\ref{BI})
\begin{eqnarray*}
\E_{\P_{\omega }}[G_{\tau ^{\ast }}(X_{\tau ^{\ast }})]-\E_{\P_{\omega
}}[G_{\tau _{M}}(X_{\tau _{M}})] &=&\newline
\E_{\P_{\omega }}\left[ (G_{1}(X_{1})-G_{2}(X_{2}))\mathbf{1}(\tau ^{\ast
}=1,\tau _{M}=2)\right] \newline
\\
&&+\E_{\P_{\omega }}\left[ (G_{2}(X_{2})-G_{1}(X_{1}))\mathbf{1}(\tau ^{\ast
}=2,\tau _{M}=1)\right] \newline
\\
&=&\E_{\P_{\omega }}\left[ |G_{1}(X_{1})-\E(G_{2}(X_{2})|\mathcal{F}_{1})|%
\mathbf{1}\{\tau _{M}\neq \tau ^{\ast }\}\right]
\end{eqnarray*}%
leads to
\begin{equation*}
\E_{\P_{\omega }}[G_{\tau ^{\ast }}(X_{\tau ^{\ast }})]-\E_{\P_{\omega
}^{\otimes M}}\left\{ \E_{\P_{\omega }}[G_{\tau _{M}}(X_{\tau _{M}})]\right\}
=\E_{\P_{\omega }^{\otimes M}}\E_{\P_{\omega }}\left[ |\Delta _{\omega }(X_{1})|%
\mathbf{1}\{\tau _{M}\neq \tau ^{\ast }\}\right]
\end{equation*}%
with $\Delta _{\omega }(x):=G_{1}(x)-C_{\omega }(x)$. By conditioning on $%
X_{1}$ we get%
\begin{eqnarray*}
\E_{\P^{\otimes M}}\E_{\P_{\omega }}\left[ |\Delta _{\omega }(X_{1})|\mathbf{1}%
\{\tau _{M}\neq \tau ^{\ast }\}\right]  &=&Am^{-\gamma /\alpha }\P(0\leq
X_{1}^{2}\leq \delta m^{-\gamma })\P_{\omega }^{\otimes M}\left( \tau
_{M}\neq \tau ^{\ast }\right)  \\
&\geq &Am^{-\gamma /\alpha }p_{\ast }\delta m^{-\gamma }\P_{\omega }^{\otimes
M}\left( \tau _{M}\neq \tau ^{\ast }\right) .
\end{eqnarray*}
Using now a well known Birg\'{e}'s or Huber's lemma, \citep[see, e.g.][p. 243]{DGL}, we get
\begin{equation*}
\sup_{\omega \in \{0,1\}^{m}}\P_{\omega }^{\otimes M}(\widehat{\tau }_{M}\neq
\tau ^{\ast })\geq \left[ 0.36\wedge \left( 1-\frac{MK_{\mathcal{H}}}{\log
(\left\vert \mathcal{H}\right\vert )}\right) \right] ,
\end{equation*}%
where $K_{\mathcal{H}}:=\sup_{P,Q\in \mathcal{H}}K(P,Q),$  \( \mathcal{H}:=\{ \P_{\omega},\,\omega\in \{ 0,1 \}^{m} \} \) and $K(P,Q)$ is a
Kullback-Leibler distance between two measures $P$ and $Q$. Since for any
two measures $P$ and $Q$ from $\mathcal{H}$ with $Q\neq P$
\begin{eqnarray*}
K(P,Q) &\leq &\sup_{\substack{ \omega _{1},\omega _{2}\in \{0,1\}^{m} \\ %
\omega _{1}\neq \omega _{2}}}\E\left[ C_{\omega
_{1}}(X_{1})\log \left\{ \frac{C_{\omega _{1}}(X_{1})}{C_{\omega _{2}}(X_{1})%
}\right\} \right.  \\
&&\left. +(1-C_{\omega _{1}}(X_{1}))\log \left\{ \frac{1-C_{\omega
_{1}}(X_{1})}{1-C_{\omega _{2}}(X_{1})}\right\} \right]  \\
&\leq &(1-G_{+}-A)^{-1}(G_{-}-A)^{-1}
\\
&&\times\P(0\leq X_{1}^{2}\leq \delta m^{-\gamma
})\left[ A^{2}m^{-2\gamma /\alpha }\right]
\\
&\leq &CMm^{-\gamma -2\gamma /\alpha
-1}
\end{eqnarray*}%
with some constant \( C>0 \) for small enough $A$, and $\log (|\mathcal{H}|)=m\log (2)$, we get%
\begin{equation*}
\sup_{\omega \in \{0,1\}^{m}}\P_{\omega }^{\otimes M}(\widehat{\tau }_{M}\neq
\tau ^{\ast })\geq \left[ 0.36\wedge \left( 1-CMm^{-\gamma -2\gamma /\alpha
-1}\right) \right] \quad
\end{equation*}%
with some constant $C>0.$ Hence,%
\begin{equation*}
\sup_{\omega \in \{0,1\}^{m}}\P_{\omega }^{\otimes M}(\widehat{\tau }_{M}\neq
\tau ^{\ast })>0
\end{equation*}
provided that $m=qM^{1/(\gamma +2\gamma /\alpha +1)}$ for small enough real number $q>0$.
In this case%
\begin{multline*}
\sup_{\omega \in \{0,1\}^{m}}\left\{ \E_{\P_{\omega }}[G_{\tau ^{\ast
}}(X_{\tau ^{\ast }})]-\E_{\P_{\omega }^{\otimes M}}\left\{ \E_{\P_{\omega
}}[G_{\tau _{M}}(X_{\tau _{M}})]\right\} \right\}
\\
\geq Ap_{\ast }\delta q^{-\gamma /\alpha-\gamma }M^{-(\gamma /\alpha
+\gamma )/(\gamma +2\gamma /\alpha +1)}=BM^{-\frac{(1+\alpha )}{2+\alpha
(1+1/\gamma )}}
\end{multline*}%
with $B=Ap_{\ast }\delta q^{-\gamma /\alpha-\gamma }.$

\section{Auxiliary results }

We have
\begin{eqnarray*}
    \Delta_{M}(\BCE)&:=&\sqrt{M}\sum_{m=1}^{M}\left\{ g_{\BCE}(X^{(m)}_{1},\ldots,X^{(m)}_{K})-\E\left[  g_{\BCE}(X_{1},\ldots,X_{K}) \right] \right\}
\end{eqnarray*}
with functions \( g_{\BCE}: \underbrace{\mathbb{R}^{d}\times\ldots\times \mathbb{R}^{d}}_{K}\to \mathbb{R} \) defined as
\[
 g_{\BCE}(x_{1},\ldots,x_{K}):=\sum_{k=0}^{K-1}G_{k+1}(x_{k+1})\mathbf{1}_{\{ x_{1}\not\in \CE_{1}, \ldots, x_{k}\not\in \CE_{k}, x_{k+1}\in \CE_{k+1} \}}.
 \]
 Denote \( \mathcal{G}=\{ g_{\BCE}: \BCE\in\mathfrak{S} \}.\)  Obviously \( \mathcal{G}\)
 is a class of uniformly bounded functions provided that all functions \( G_{k} \) are uniformly bounded.
 \paragraph{Definition} Let \( \mathcal{N}_{B}(\delta, \mathcal{G},\P) \) be the
smallest value of \( n \) for which there exist pairs of functions \( \{ [g_{j}^{L},g_{j}^{U}] \}_{j=1}^{n} \)
such that \( \| g_{j}^{U}-g_{j}^{L }\|_{L_{2}(\P)}\leq \delta   \) for all \( j=1,\ldots, n, \) and such that
for each \( g\in \mathcal{G}, \) there is \( j=j(g)\in \{ 1,\ldots,n \} \) such that
\begin{eqnarray*}
    g_{j}^{L}\leq g \leq g_{j}^{U}.
\end{eqnarray*}
Then \( \mathcal{H}_{B}(\delta, \mathcal{G},\P)=\log \left[ \mathcal{N}_{B}(\delta, \mathcal{G},\P) \right] \)
is called the  entropy with bracketing of \( \mathcal{G} \).
The following theorem follows directly from Theorem 5.11 in \citet{V}.
\begin{thm}
\label{EINEQ}
Assume that there exists a constant \( A>0 \) such that
\begin{eqnarray}
\label{CA1}
\mathcal{H}_{B}(\delta,\mathcal{G},\P)\leq A\delta^{-\nu}
\end{eqnarray}
for any \( \delta>0 \) and some \( \nu>0 \), where
\( \mathcal{H}_{B}(\delta,\mathcal{G},\P) \) is the \( \delta \)-entropy with bracketing
of \( \mathcal{G}. \) Fix some \( \BCE_{0}\in \mathfrak{S} \) then for any \( \varepsilon   \geq  M^{-1/(2+\nu )}  \)
\begin{eqnarray*}
    &&\P\left( \sup_{\BCE\in \mathfrak{S},\,\|g_{\BCE}-g_{\BCE_{0}}\|_{L_{2}(\P)}\leq\varepsilon  }|\Delta_{M}(\BCE)-\Delta_{M}(\BCE_{0})|>U\varepsilon ^{1-\frac{\nu }{2}}  \right)
    \leq C \exp(-U\varepsilon ^{-\nu }/C^{2} ),
\\
    &&\P\left(  \sup_{\BCE\in \mathfrak{S},\,\|g_{\BCE}-g_{\BCE_{0}}\|_{L_{2}(\P)}\leq\varepsilon  }\frac{|\Delta_{M}(\BCE)-\Delta_{M}(\BCE_{0})|}{\|g_{\BCE}-g_{\BCE_{0}}\|_{L_{2}(\P)}^{1-\nu/2 }}>U  \right)
    \leq C \exp(-U/C^{2} ).
\end{eqnarray*}
for all \( U>C \) and \( M>M_{0}, \) where \( C \) and \( M_{0} \) are two positive constants.
Moreover, for any \( z>0 \)
\begin{eqnarray*}
    \P\left( \sup_{\BCE\in \mathfrak{S}}|\Delta_{M}(\BCE)-\Delta_{M}(\BCE_{0})|>z\sqrt{M}  \right)
    \leq C \exp(-Mz^{2}/C^{2}B )
\end{eqnarray*}
with some positive constant \( B>0 \).
\end{thm}
Let us define a pseudedistance \( \Delta_{X} \) between any two sets \( \BCE, \BCE' \in \mathfrak{B} \) in the following way
\begin{multline*}
\Delta_{X}(\mathcal{S}_{1}\times\ldots\times\mathcal{S}_{K},\mathcal{S}'_{1}\times\ldots\times \mathcal{S}'_{K})
:=\sum_{k=1}^{K}\P\left(X_{k}\in (\mathcal{S}_{k}\triangle \mathcal{S}'_{k})\setminus \left(\bigcap_{l=k}^{K-1} \mathcal{S}'_{l}\right)\right).
\end{multline*}
It obviously holds \( \Delta_{X}(\BCE,\BCE')\leq d_{X}(\BCE,\BCE') \) for the pseudodistance \( d_{X} \) defined in
The following Lemma will be frequently used in the sequel.
\begin{lem}
\label{BI}
For any \( \BCE,\BCE'\in \mathfrak{B} \) it holds with probability one
\begin{multline}
\label{BasicIneqE}
\left| G_{\tau_{k}(\BCE)}
\left(X_{\tau_{k}(\BCE)}\right)-G_{\tau_{k}(\BCE')}
\left(X_{\tau_{k}(\BCE')}\right) \right|
\\
\leq
\sum_{l=k}^{K-1}|G_{l}(X_{l})-G_{\tau_{l+1}(\BCE)}(X_{\tau_{l+1}(\BCE)})|\mathbf{1}_{\left\{X_{l}\in (\mathcal{S}_{l}\triangle \mathcal{S}'_{l})\setminus \left(\bigcap_{l'=l}^{K-1} \mathcal{S}'_{l'}\right)
\right\}}
\end{multline}
and
\begin{multline}
\label{BasicIneq}
V^{*}_{k}(X_{k})-\E\left[G_{\tau_{k}(\BCE)}(X_{\tau_{k}(\BCE)})|\mathcal{F}_{k}\right]
\\
=\E\left[ \left.\sum_{l=k}^{K-1}\left|G_{l}(X_{l})-\E[V^{*}_{l+1}(X_{l+1})|\mathcal{F}_{l}]\right|\mathbf{1}_{\left\{X_{l}\in (\mathcal{S}^{*}_{l}\triangle \mathcal{S}_{l})\setminus \left(\bigcap_{l'=l}^{K-1} \mathcal{S}_{l'}\right) \right\}}
\right|\mathcal{F}_{k}\right]
\end{multline}
for \( k=1,\ldots,K-1 \).
\end{lem}
\begin{proof}
We prove \eqref{BasicIneq} by induction. The inequality \eqref{BasicIneqE} can be proved in a similar way. For \( k=K-1 \)
we get
\begin{multline}
\label{BasiqIneq2}
V^{*}_{K-1}(X_{K-1})-V_{K-1}(X_{K-1})=
\\
\nonumber
=
\E\left[\left.(G_{K-1}(X_{K-1})-G_{K}(X_{K}))
\mathbf{1}_{\{\tau^{*}_{K-1}=K-1,\,\tau_{K-1}=K\}}\right|\mathcal{F}_{K-1}\right]
\\
\nonumber
+\E\left[\left.(G_{K}(X_{K})-G_{K-1}(X_{K-1}))
\mathbf{1}_{\{\tau^{*}_{K-1}=K,\,\tau_{K-1}=K-1\}}\right|\mathcal{F}_{K-1}\right]
\\
\nonumber
=|G_{K-1}(X_{K-1})-\E [G_K(X_{K})|\mathcal{F}_{K-1}]|\mathbf{1}_{\{\tau_{K-1}\neq\tau^{*}_{K-1}\}}
\end{multline}
since events \( \{ \tau^{*}_{K-1}=K \} \) and \( \{\tau_{K-1}=K \} \) are measurable w.r.t. \( \mathcal{F}_{K-1} \)
  and \( G_{K-1}(X_{K-1})\geq  \E [G_K(X_{K})|\mathcal{F}_{K-1}]\) on the set \( \{ \tau^{*}_{K-1}=K-1 \}. \) Thus, \eqref{BasicIneq} holds with \( k=K-1 \). Suppose
that \eqref{BasicIneq} holds with
\( k=K'+1 \). Let us prove it for \( k=K' \). Consider a decomposition
\begin{eqnarray}
\label{DecompBI}
G_{\tau^{*}_{K'}}(X_{\tau^{*}_{K'}})-G_{\tau_{K'}}(X_{\tau_{K'}})&=& S_{1}+S_{2}+S_{3}
\end{eqnarray}
with
\begin{eqnarray*}
S_{1}&:=&\left( G_{\tau^{*}_{K'}}(X_{\tau^{*}_{K'}})-G_{\tau_{K'}}(X_{\tau_{K'}}) \right)
\mathbf{1}_{\{\tau^{*}_{K'}>K',\,\tau_{K'}>K'\}},
\\
S_{2}&:=&\left( G_{\tau^{*}_{K'}}(X_{\tau^{*}_{K'}})-G_{\tau_{K'}}(X_{\tau_{K'}}) \right)
\mathbf{1}_{\{\tau^{*}_{K'}>K',\,\tau_{K'}=K'\}},
\\
S_{3}&:=&\left( G_{\tau^{*}_{K'}}(X_{\tau^{*}_{K'}})-G_{\tau_{K'}}(X_{\tau_{K'}}) \right)
\mathbf{1}_{\{\tau^{*}_{K'}=K',\,\tau_{K'}>K'\}}.
\end{eqnarray*}
Using the fact that \( \tau_{k}=\tau_{k+1}   \) if \( \tau_{k}>k  \) for any \( k=1,\ldots,K-1 \), we get
\begin{eqnarray*}
\E\left[S_{1}|\mathcal{F}_{K'}\right]
&=&
\E\left[ \left.\left( V^{*}_{K'+1}(X_{K'+1})-V_{K'+1}(X_{K'+1}) \right)
\mathbf{1}_{\{\tau^{*}_{K'}>K',\,\tau_{K'}>K'\}}\right|\mathcal{F}_{K'}\right],
\\
\E\left[S_{2}|\mathcal{F}_{K'}\right]&=&
\left( \E\left[ \left. G_{\tau^{*}_{K'+1}}(X_{\tau^{*}_{K'+1}})\right | \mathcal{F}_{K'} \right]
 -G_{K'}(X_{K'}) \right)
\mathbf{1}_{\{\tau^{*}_{K'}>K',\,\tau_{K'}=K'\}}
\\
&=&\left(\E\left[ \left. V^{*}_{K'+1}(X_{K'+1})\right | \mathcal{F}_{K'} \right] -G_{K'}(X(t_{K'})) \right)
\mathbf{1}_{\{\tau^{*}_{K'}>K',\,\widehat\tau_{K'}=K'\}}
\end{eqnarray*}
and
\begin{eqnarray*}
\E\left[S_{3}|\mathcal{F}_{K'}\right]&=&
\left( G_{K'}(X_{K'})-\E\left[ G_{\tau_{K'+1}}(X_{\tau_{K'+1}})|\mathcal{F}_{K'}\right]  \right)
\mathbf{1}_{\{\tau^{*}_{K'}=K',\,\tau_{K'}>K'\}}
\\
&=&\left(G_{K'}(X_{K'})-\E[V^{*}_{K'+1}(X_{K'+1})|\mathcal{F}_{K'}]\right)
\mathbf{1}_{\{\tau^{*}_{K'}=K',\,\tau_{K'}>K'\}}
\\
&&+\E\left[ \left.\left( V^{*}_{K'+1}(X_{K'+1})-V_{K'+1}(X_{K'+1}) \right)
\mathbf{1}_{\{\tau^{*}_{K'}=K',\,\tau_{K'}>K'\}}\right|\mathcal{F}_{K'}\right],
\end{eqnarray*}
with probability one. Hence
\begin{eqnarray*}
V^{*}_{K'}(X_{K'})-V_{K'}(X_{K'})&=&
\left|G_{K'}(X_{K'})-\E[V^{*}_{K'+1}(X_{K'+1})|\mathcal{F}_{K'}]\right|
\mathbf{1}_{\{\tau_{K'}\neq\tau^{*}_{K'}\}}
\\
&&+\E\left[\left.\left( V^{*}_{K'+1}(X_{K'+1})-V_{K'+1}(X_{K'+1}) \right)
\right|\mathcal{F}_{K'}\right]\mathbf{1}_{\{\tau_{K'}>K'\}}
\end{eqnarray*}
since \( G_{K'}(X_{K'})-\E[V^{*}_{K'+1}(X_{K'+1})\geq 0 \) on the set \( \{ \tau^{*}_{K'}=K' \}. \)
Our induction assumption implies now that
\begin{multline*}
V_{K'}^{*}(X_{K'})-V_{K'}(X_{K'})=
\\
\E\left[ \sum_{k=K'}^{K-1}|G_{l}(X_{l})-\E[V^{*}_{l+1}(X_{l+1})|\mathcal{F}_{l}]|\mathbf{1}_{\{\tau_{k}\neq \tau^{*}_{k},\tau_{k}>k,\ldots,\tau_{K-1}>K-1\}}|\mathcal{F}_{K'}\right]
\end{multline*}
and hence \eqref{BasicIneq} holds with \( k=K' \).
\end{proof}
\begin{cor}
\label{DFX}
If  \( \max_{k=1,\ldots,K}\| G_{k} \|_{\infty}<A_{G} \)  with some constant \( A_{G}>0 \), then
\begin{eqnarray*}
    \Delta_{G}(\BCE,\BCE')\leq 2A_{G}\sqrt{2\Delta_{X}(\BCE,\BCE')}
\end{eqnarray*}
for any \( \BCE,\BCE'\in \mathfrak{B}. \)
\end{cor}
\begin{proof}
Follows directly from \eqref{BasicIneqE} since \( G_{\tau}(X_{\tau}) \leq A_{G} \) a.s.
for any stopping time \( \tau  \) taking values in \( \{ 1,\ldots, K \}. \)
\end{proof}

\begin{cor}
\label{DDX}
 Assume that \eqref{BA} holds for \( \delta < \delta_{0}<1/2  \), then there exist constants \( \upsilon_{\alpha} \) and
\( \delta_{\alpha} \) such that
\begin{eqnarray}
\label{BAD}
    \Delta(\BCE)\geq \upsilon_{\alpha}\Delta^{(1+\alpha)/\alpha}_{X}(\BCE^{*},\BCE)
\end{eqnarray}
for all \( \BCE \in \mathfrak{B}\) satisfying \( \Delta_{X}(\BCE^{*},\BCE)\leq \delta_\alpha \). Moreover
it holds
\begin{eqnarray}
\label{BAD1}
\Delta_{X}(\BCE^{*},\BCE)\leq \left( \frac{2^{1/\alpha}}{\delta_{0}} \right)
 \Delta(\BCE)+\frac{\delta_{\alpha}}{2(1+\alpha)}.
\end{eqnarray}
for any  \( \BCE \in \mathfrak{B}.\)
\end{cor}
\begin{proof}
For any \( \delta\leq\delta_{0} \) define the sets
\[
\mathcal{A}_{k}:=\left\{x\in \mathbb{R}^{d}:\left|\E[V^{*}_{k+1}(X_{k+1})|X_{k}=x]-G_{k}(x)\right|>\delta \right\},\quad
k=1,\ldots,K-1.
\]
Due to \eqref{BasicIneq} we have
\begin{eqnarray}
\notag
\Delta(\BCE)&\geq& \delta\sum_{k=1}^{K-1}\P\left(X_{k}\in (\mathcal{S}^{*}_{k}\triangle \mathcal{S}_{k})\setminus \left(\bigcap_{l=k}^{K-1}
\mathcal{S}_{k}\right)\bigcap\mathcal{A}_{k}\right)
\\
\notag
&\geq&
\delta\sum_{k=1}^{K-1}\left\{ \P\left(X_{k}\in (\mathcal{S}^{*}_{k}\triangle \mathcal{S}_{k})\setminus \left(\bigcap_{l=k}^{K-1}
\mathcal{S}_{k}\right)\right)-\P(\bar{\mathcal{A}}_{k}) \right\}
\\
\label{BADP1}
&\geq& \delta [\Delta_{X}(\BCE^{*},\BCE)-A_{0}\delta^{\alpha}]
\end{eqnarray}
with \( A_{0}=\sum_{k=1}^{K-1}A_{k,0},  \) where \( A_{k,0} \) were defined in \eqref{BA}.
The maximum of \eqref{BADP1} is attained at \( \delta^{*}=[\Delta_{X}(\BCE^{*},\BCE)/(\alpha+1)A_{0}]^{1/\alpha}
\). Since \( \delta^{*}\leq\delta_{0} \) for
\( \Delta_{X}(\BCE^{*},\BCE)\leq A_{0}(\alpha+1)\delta^{\alpha}_{0} \)
the inequality  \eqref{BAD} holds with \( \upsilon_{\alpha}:=A_{0}^{-1/\alpha}\alpha(1+\alpha)^{-1-1/\alpha} \)
and \( \delta_{\alpha}:=A_{0}(\alpha+1)\delta^{\alpha}_{0}\). The inequality \eqref{BAD1} follows directly from
\eqref{BADP1} by taking \( \delta = \delta_{0}/2^{1/\alpha}.  \)
\end{proof}


\begin{thebibliography}{99}                                                                                               %


\bibitem[Andersen(2000)]{A} L. Andersen (2000). A simple approach to the pricing of Bermudan
swaptions in the multi-factor Libor Market Model. \textit{Journal of
Computational Finance}, \textbf{3}, 5-32.
\bibitem[Belomestny(2009)]{B} D. Belomestny (2009). Pricing Bermudan options using nonparametric regression: optimal rates of convergence for lower estimates, http://arxiv.org/abs/0907.5599, forthcoming in \textit{Finance and Stochastics}.
\bibitem[Broadie and Glasserman(1997)]{BG} M. Broadie and P. Glasserman (1997). Pricing American-style
securities using simulation. \textit{J. of Economic Dynamics and Control}, \textbf{21}%
, 1323-1352.

\bibitem[Broadie and Detemple(1997)]{BD} M. Broadie and J. Detemple (1997). The valuation
of American options on multiple assets. \textit{Mathematical Finance}, \textbf{7}(3), 241-286.

\bibitem[Carriere(1996)]{Car} J. Carriere (1996). Valuation of early-exercise price of
options using simulations and nonparametric regression. \textit{Insuarance:
Mathematics and Economics}, \textbf{19}, 19-30.

\bibitem[Devroye, Gy\"orfi and Lugosi(1996)]{DGL} L. Devroye, L. Gy\"orfi and G. Lugosi (1996).
A probabilistic theory of pattern recognition. Application of Mathematics (New York), \textbf{31}, Springer.


\bibitem[Dudley(1999)]{D} R.M. Dudley (1999). \textsl{Uniform central limit theorems}. Cambridge University Press.


\bibitem[Egloff(2005)]{E} D. Egloff (2005). Monte Carlo algorithms for optimal stopping
and statistical learning. \textit{Ann. Appl. Probab.}, \textbf{15},  1396-1432.

\bibitem[Egloff, Kohler and Todorovic(2007)]{EKT} D. Egloff, M. Kohler and N. Todorovic (2007).
A dynamic look-ahead Monte Carlo algorithm for pricing Bermudan options,
\textit{Ann. Appl. Probab.}, \textbf{17}, 1138-1171.

\bibitem[Garcia(2001)]{G} D. Garcia (2001). Convergence and biases of Monte Carlo estimates of American option prices
using a parametric exercise rule. Working paper.

\bibitem[Glasserman(2003)]{Gl}P. Glasserman (2003). Monte Carlo Methods in Financial
Engineering. Springer.

\bibitem[Kleywegt, Shapiro and Homem-de-Mello(2001)]{KSH} A.J. Kleywegt,  A. Shapiro  and T. Homem-de-Mello (2001). The sample average approximation method for stochastic discrete optimization, \textit{SIAM J. Optim.}, \textbf{12}, 479-502.

\bibitem[Longstaff and Schwartz(2001)]{LS} F. Longstaff and E. Schwartz (2001). Valuing American options
by simulation: a simple least-squares approach. \textit{Review of Financial Studies},
\textbf{14}, 113-147.

\bibitem[Mammen and Tsybakov(1999)]{MT} E. Mammen and A. Tsybakov (1999). Smooth discrimination analysis. \textit{Ann. Statist.,} \textbf{27}, 1808-1829.

\bibitem[Peskir and Shiryaev(2006)]{PS} G. Peskir and
A. Shiryaev (2006). \textsl{Optimal Stopping and
Free-Boundary Problems}.  LM - Lectures in Mathematics ETH Zürich.


\bibitem[Shapiro(1993)]{S0} A. Shapiro (1993). Asymptotic behavior of optimal solutions in stochastic programming, \textit{Math. Oper. Res.}, \textbf{18}, 829-845.

\bibitem[Shapiro and Homem-de-Mello(2000)]{SH} A. Shapiro and T. Homem-de-Mello (2000). On the rate of convergence of optimal solutions of Monte Carlo
approximations of stochastic programs. \textit{SIAM J. Optim.}, \textbf{11}(1), 70-86.

\bibitem[Snell (1952)]{SN} J. L. Snell (1952). Applications of martingale system theorems. \textit{Trans.
Amer. Math. Soc}. \textbf{73}, 293--312.


\bibitem[Tsitsiklis and Van Roy(1999)]{TV} J. Tsitsiklis and B. Van Roy (1999). Regression methods for
pricing complex American style options. \textit{IEEE Trans. Neural. Net.}, \textbf{12}%
, 694-703.

\bibitem[Van de Geer(2000)]{V} S. Van de Geer (2000). \textsl{Applications of Empirical Process Theory}. Cambridge Univ. Press.

\end{thebibliography}
\end{document}